\documentclass{amsart}
\usepackage{amssymb,amsmath,amsthm,mathrsfs} 
\usepackage[backend=biber, style=alphabetic, doi=false, url=false, isbn=false, giveninits=true, maxbibnames=99]{biblatex}
\usepackage{graphicx} 
\usepackage{verbatim}
\usepackage{tikz-cd} 
\usepackage{tikz}
\usepackage{hyperref}
\usepackage{geometry}

\usepackage{booktabs}
\usepackage{comment}

\usepackage{amsthm}
\usepackage{amssymb}
\usepackage{mathtools}
\usepackage{amsmath}
\usepackage{mathrsfs}
\usepackage{xcolor}
\usepackage{hyperref}

\usepackage{calligra}
\usepackage{tikz-cd}
\usepackage{graphicx} 
\usepackage{verbatim}
\usepackage{tikz-cd} 
\usepackage{tikz}
\usepackage{tabularx}

\usepackage{mathtools}
\usepackage{verbatim}
\usepackage{tikz-cd} 
\usepackage{tikz}
\usepackage{hyperref}
\usepackage{tabularx}

\usepackage{mathtools}
\usepackage{float}
\usepackage{booktabs}

\usepackage{enumerate}

\newtheorem{theorem}{Theorem}[section]
\newtheorem{corollary}[theorem]{Corollary}
\newtheorem{lemma}[theorem]{Lemma}
\newtheorem{proposition}[theorem]{Proposition}
\newtheorem{definition}[theorem]{Definition}
\newtheorem{remark}[theorem]{Remark}

\newtheorem{example}[theorem]{Example}

\newtheorem{claim}[theorem]{Claim}

\newcommand{\Z}{\mathbb{Z}}
\newcommand{\Q}{\mathbb{Q}}

\newcommand{\OO}{\mathcal{O}}

\newcommand{\G}{\mathscr{G}}
\newcommand{\C}{\mathfrak{C}}

\newcommand{\im}{\text{Im}}

\newcommand{\p}{\mathfrak{p}}

\newcommand{\PP}{\mathfrak{P}}

 \DeclareFontFamily{U}{wncy}{}
  \DeclareFontShape{U}{wncy}{m}{n}{<->wncyr10}{}
   \DeclareSymbolFont{mcy}{U}{wncy}{m}{n}   
  \DeclareMathSymbol{\Sh}{\mathord}{mcy}{"58}
\DeclareMathOperator{\Ima}{Im}

\DeclareMathOperator{\Hom}{Hom}

\DeclareMathOperator{\Ass}{Ass}
\DeclareMathOperator{\Ht}{ht}
\DeclareMathOperator{\Ann}{Ann}
\DeclareMathOperator{\ch}{Ch}
\DeclareMathOperator{\Gal}{Gal}
\DeclareMathOperator{\coker}{coker}
\DeclareMathOperator{\Ext}{Ext}
\DeclareMathOperator{\Tor}{Tor}
\DeclareMathOperator{\depth}{depth}
\DeclareMathOperator{\Spec}{Spec}
\DeclareMathOperator{\Supp}{Supp}
\DeclareMathOperator{\grade}{grade}
\DeclareMathOperator{\Res}{Res}

\newcommand{\A}{\mathbb{A}}

\newcommand{\I}{\mathbb{I}}

\newcommand{\D}{\mathfrak{D}}

\title{On the Quotient of a Pseudo-null Module}

\author{Peikai Qi}
\email{qipeikai@amss.ac.cn}
\address{Morningside Center of Mathematics, Chinese Academy of Science, Beijing}
\date{\today}
\thanks{Thanks to Ruichen Xu, Jie Yang, and Preston Wake for their insightful discussions. The paper grows out of the author's thesis. }

\begin{document}
\maketitle

\begin{abstract}
Motivated by questions arising in noncommutative Iwasawa theory, let $R$ be a (not necessarily commutative) ring, $T \in R$ a regular central element, and $M$ a pseudo-null $R$-module. We investigate necessary and sufficient conditions under which the quotient $M/TM$ is pseudo-null as an $R/TR$-module. 

We first give necessary and sufficient conditions in terms of associated prime ideals when $R$ is commutative. Then we apply this to the case when $R$ is a Krull domain and obtain a precise relationship between the characteristic ideals of $M/TM$ and the $T$-torsion submodule $M[T]$. We give necessary and sufficient conditions in terms of $\Ext$ groups when $R$ is a noncommutative ring and then compare them with the criterion when $R$ is commutative. Lastly, we give an application to noncommutative Iwasawa theory. Roughly speaking, if `big' dual fine Selmer group is pseudo-null, then `most' specialized dual fine Selmer group is pseudo-null.
\end{abstract}
\section{Introduction}

Pseudo-null modules are negligible for characteristic ideals over a
commutative Iwasawa algebra, but they need not remain so after
quotient. Quotienting by a regular element can turn a component
of codimension two into one of codimension one. Consequently, a
pseudo-null module may contribute a nontrivial characteristic ideal
after quotient. 
\begin{definition}\label{def}
Let \(R\) be a commutative ring. A finitely generated \(R\)-module
\(M\) is \emph{pseudo-null} if \(M_P=0\) for every prime ideal \(P\)
with \(\Ht_R(P)\leq 1\), or equivalently if
\(\Ht_R(\Ann_R(M))\geq 2\).
\end{definition}

Write \(\Ass_R(M)\) for the set of associated primes of \(M\), and
\(M[T]=\{m\in M:Tm=0\}\).

\begin{theorem}\label{ht=2 intro}
Let \(R\) be a Noetherian commutative ring, let \(T\in R\) be a
regular element, and let \(M\) be a finitely generated pseudo-null
\(R\)-module. Then \(M/TM\) is \textbf{not} pseudo-null over \(R/TR\) if and
only if there exists \(P\in\Ass_R(M)\) such that $\Ht_R(P)=2
$ and $T\in P.$
\end{theorem}

In particular, \(T\) need not act injectively on \(M\): only the
height-two associated primes obstruct descent, while associated
primes of larger height do not. The criterion applies to arbitrary
Noetherian commutative rings, without regularity or Cohen--Macaulay
hypotheses on \(R\). Its proof is inspired by
\cite{ozaki30iwasawa} and \cite[Lemma~3.5]{Ochiai2005}.

When characteristic ideals are defined, we can also measure the
height-one contribution created by quotient.
\begin{proposition}
Assume that \(R\) and \(R/TR\) are Noetherian Krull domains.
For every finitely generated
pseudo-null \(R\)-module \(M\), one has
\[
\ch_{R/TR}(M/TM)=\ch_{R/TR}(M[T]).
\]
\end{proposition}

In particular, the height-one contribution of \(M/TM\) is measured
by the \(T\)-torsion of \(M\). Section~3 gives the
corresponding formula for a finitely generated torsion module whose
quotient remains torsion. These results extend
\cite[Propositions~2.7 and~2.10]{bandini2013characteristic} beyond
the power-series setting \(R=A[[T]]\).

For noncommutative rings, we use a homological definition of
pseudo-nullity.

\begin{definition}
Let \(R\) be a ring. A right \(R\)-module \(M\) is \emph{pseudo-null}
if \(\Ext_R^i(M',R)=0\) for \(i=0,1\) and every submodule
\(M'\subseteq M\).
\end{definition}

For finitely generated modules over a Noetherian commutative ring
satisfying Serre's condition \((S_2)\), this agrees with
Definition~\ref{def}; see Proposition~\ref{Sn}. The following
criterion replaces the associated-prime condition by an
\(\Ext^2\)-obstruction.

\begin{theorem}\label{intro non ht=2}
Let \(R\) be a ring, let \(T\in R\) be a regular central element,
and let \(M\) be a pseudo-null right \(R\)-module. Then \(M/MT\)
is \textbf{not} pseudo-null over \(R/TR\) if and only if there exists a
submodule \(M'\) with $ MT\subseteq M'\subseteq M $ such that $
\Ext_R^2(M',R)[T]\neq 0$.
\end{theorem}

Theorem~\ref{non ht=2} proves higher-\(\Ext\) analogues without
Noetherian hypotheses. The criterion allows \(T\) to be a zero
divisor on \(M\), unlike the \(M\)-regularity hypothesis in the
graded Auslander--Gorenstein descent result
\cite[Corollary~4.4]{Levasseur1992}.
The connection with the commutative criterion is made explicit in
Proposition~\ref{ext ht}: if \(R\) is Noetherian, commutative, and
Cohen--Macaulay, and \(M\) is finitely generated and pseudo-null, then
\[
\Ass_R\!\bigl(\Ext_R^2(M,R)\bigr)
=
\{P\in\Ass_R(M):\Ht_R(P)=2\}.
\]
Thus the homological criterion recovers precisely the height-two
obstruction in Theorem~\ref{ht=2 intro}.

We now describe the arithmetic application. Let \(p\geq 3\), let
\(K\) be a number field, and let \(\I\) be a commutative complete
regular local ring of characteristic zero with finite residue field
of characteristic \(p\). Let \(\mathbb T\) be a finite free
\(\I\)-module with a continuous \(\I\)-linear action of \(G_K\), and
put
$\A=\mathbb T\otimes_\I
\Hom_{\Z_p,\mathrm{cts}}(\I,\Q_p/\Z_p)$.
Choose a finite set \(S\) containing the primes above \(p\), the
archimedean primes, and all primes where \(\mathbb T\) is ramified,
and let \(K_S\) be the maximal extension of \(K\) unramified outside
\(S\). Let \(L/K\) be a Galois extension contained in \(K_S\) and
containing the cyclotomic \(\Z_p\)-extension \(K^{\mathrm{cyc}}/K\).
Assume that \(G=\Gal(L/K)\) is a compact pro-\(p\) \(p\)-adic Lie
group without \(p\)-torsion, and put \(H=\Gal(L/K^{\mathrm{cyc}})\).

A \emph{specialization} is a local homomorphism
\(\phi:\I\to\OO_\phi\), where \(\OO_\phi\) is a commutative complete
regular local ring of characteristic zero, finite as an \(\I\)-module.
Write \(\PP_\phi=\ker\phi\) and
$\A_\phi=\mathbb T\otimes_\I
\Hom_{\Z_p,\mathrm{cts}}(\OO_\phi,\Q_p/\Z_p)$.
The notation \(Y_S(\A/L)\) denotes the Pontryagin dual of the fine
Selmer group, and similarly for \(\A_\phi\); see
Section~\ref{fine sel gp}.

\begin{theorem}\label{intro fine specialization}
Assume that \(Y_S(\A/L)\) is finitely generated over \(\I[[H]]\). Then \(Y_S(\A/L)\)  is pseudo-null over \(\I[[G]]\) if and only if there exists a Zariski
dense  subset \(U\subseteq\Spec(\I)\) such that
\(Y_S(\A_\phi/L)\) is pseudo-null over \(\OO_\phi[[G]]\) for every
specialization \(\phi\) with \(\PP_\phi\in U\).
\end{theorem}

This is a combination of Theorem \ref{generalize Jha} and Corollary \ref{main converse} in the paper. Most results of this type require a dimension hypothesis such as
$(\mathrm{Dim}_S)$, which requires the $p$-adic Lie extension $G$ to have
dimension at least two and the decomposition group
$G_v$ to satisfy $\dim G_v\geq 2$ for every $v\in S$.
For example, Jha's specialization theorem
\cite[Theorem~10]{Jha2012} assumes $(\mathrm{Dim}_S)$, as well as
$(\mathrm{Irr})$, requiring the residual Galois representation to be
absolutely irreducible and the field \(K\) is
abelian over \(\Q\).  In contrast, our result does not require
$(\mathrm{Dim}_S)$ or any irreducibility or large-image hypothesis on
the Galois representation or abelian condition on field extension $K/\Q$, and it applies to Galois representations
over a general complete regular local coefficient ring. 

The proof was separated into two steps. In particular, the first step, Theorem~\ref{generic control},
shows that the natural map $Y_S(\A/L)\otimes_\I\OO_\phi
\longrightarrow
Y_S(\A_\phi/L)$
has pseudo-null kernel and cokernel for Zariski generic
specializations. This statement requires neither pseudo-nullity of
\(Y_S(\A/L)\) nor its finite generation over \(\I[[H]]\).

 Under the assumption Hecke algebra $\I\cong \OO[[W]]$, Coates--Sujatha prove in \cite[Lemma 6.1]{coates_sujatha_2012} that, under certain conditions, if $M/(t_\xi M)$ is pseudo-null over $\OO_\xi[[G]]$ for an arithemetic specialization $\xi$, then $M$ is pseudo-null over $\I[[G]]$. Our Theorem \ref{intro non ht=2} deals with the converse direction. We assume first that $M$ is pseudo-null and then find a necessary and sufficient condition for the quotient $M/(t_\xi M) $ to be pseudo-null.
Independently, Lemma~\ref{lemma complementary} also establishes the complementary algebraic ascent statement: pseudo-nullity of a single
\(M\otimes_\I\OO_\phi\) implies pseudo-nullity of \(M\) when assuming  \(M\) is finitely generated over \(\I[[H]]\). This extends
\cite[Lemma~6.1]{coates_sujatha_2012} to the specializations for general ring $\I$ considered here. 

Finally, the exact quotient criterion gives more information than
generic descent for height-one coefficient primes. We apply Theorem \ref{intro non ht=2} to a Hida family to give a new proof of  \cite[Theorem 10]{Jha2012}. 
Proposition~\ref{Jha} shows that, if \(M\) is pseudo-null over
\(\Lambda=\I[[G]]\) and finitely generated over \(\I[[H]]\), then
\(M/M\PP\) is pseudo-null over \(\Lambda/\PP\Lambda\) for all but
finitely many height-one primes \(\PP\subset\I\). The exceptional
primes are exactly those satisfying
$\Ext_\Lambda^2(M,\Lambda)[\PP]\neq 0$.
When \(G\cong\Z_p^d\), the number of exceptional primes is bounded by the number of
height-two associated primes of \(M\).

\section{Proof of the main theorem for commutative rings}
We prove the following theorem in this section. 

\begin{theorem} \label{ht=2}
    Let $R$ be a Noetherian commutative ring and let $T$ be a regular element in $R$. Let $M$ be a finitely generated $R$-module. Assume $M$ is a pseudo-null $R$-module. Then $M/TM$ is \textbf{not} pseudo-null as an $R/T$-module if and only if $T\in P$, where $P$ is an associated prime of $M$ and $\Ht_R(P)=2$.
\end{theorem}

We begin with some definitions and lemmas.

\begin{definition}
    Define $\D_R^n$ as the subcategory 
    of the category of finitely generated $R$-modules consisting of modules $M$ such that $\Ht(\Ann_R(M))\geq n+1$.
\begin{equation*}
  \begin{array}{rl}
     \D_R^n:=&\{M|R \text{ module } M \text{ such that } M_P=0 \text{ for all prime ideals } P \text{with } ht(P)\leq n \} \\
         :=&\{M|R \text{ module } M \text{ such that } \Ht(\Ann_R(M))\geq n+1 \}.
    \end{array}  
\end{equation*}
    \end{definition}

We see that an $R$-module $M$ belongs to $\D_R^1$ if and only if $M$ is a pseudo-null module. The following property follows easily from the fact that localization is an exact functor.

\begin{lemma}\label{property}
    Let  $   0\rightarrow M'\rightarrow M\rightarrow M''\rightarrow 0$
   be an exact sequence of $R$-modules. Then $M $ is in $\D_R^n$ if and only if both $M'$ and $M''$ are in $\D_R^n$.
\end{lemma}

\begin{lemma}\label{quotient}
    Let $T$ be a regular element of $R$, and let $M$ be an $R$-module. Then 
    \[
    M/TM\in \D_{R/T}^n \quad \text{if and only if} \quad M/TM \in \D_R^{n+1}.
    \]
\end{lemma}

\begin{proof}
    Let $P$ be a prime ideal in $R$ containing $T$, and let $\bar{P}$ be its image in $R/T$. We have $\Ht(P) = \dim(R_P)$ and $\Ht(\bar{P}) = \dim((R/T)_{\bar{P}}) = \dim(R_P/T)$. By \cite[(15.F) Lemma 4]{matsumura1970commutative}, since $T$ is a regular element of $R$, we have $ \dim(R_P/T) = \dim(R_P) - 1$.
   This implies 
    \[
    \Ht_R(P) = \Ht_{R/T}(\bar{P}) + 1.
    \]
    The conclusion now follows from the definition.
\end{proof}

We will frequently use the above two lemmas. Now we begin the proof of Theorem \ref{ht=2}. Let $R$ be a Noetherian ring, and let $M$ be a finitely generated $R$-module. Consider the shortest primary decomposition of the zero submodule in the Noetherian $R$-module $M$:
\[
0 = \bigcap_{i=1}^r Y_i.
\]
In other words, $\Ass_R(M/Y_i) = \{P_i\}$, where $P_i \neq P_j$ for $i \neq j$. By \cite[Lemma 8.E]{matsumura1970commutative}, we have the set of associated primes $
\Ass_R(M) = \{P_1, P_2, \dots, P_r\}$.

\begin{lemma}
    Let $R$ be a Noetherian ring and $M$ be a finitely generated $R$-module. Assume $M\in \D_R^1$ and consider the following short exact sequence.
    \[
    0\rightarrow M\xrightarrow{\phi}\bigoplus_{i=1}^rM/Y_i\rightarrow D\rightarrow 0
    \]
    where $\phi$ is induced by the natural projection and $D$ is the cokernel of the map $\phi$. Then $D$ is in $ \D_R^2$.
\end{lemma}
\begin{proof}
    An element of $\bigoplus_{i=1}^r(\cap_{j\neq i}Y_j+Y_i)/Y_i$ can be written as $(\bar{a}_1,\bar{a}_2,\cdots,\bar{a}_r)$ for elements $a_i\in  \cap_{j\neq i}Y_j$. Notice that 
    \[
    \phi(a_1+a_2+\cdots+a_r)=(\bar{a}_1,\bar{a}_2,\cdots,\bar{a}_r)
    \]
Hence, 
\[
\bigoplus_{i=1}^r(\cap_{j\neq i}Y_j+Y_i)/Y_i\subset \Ima(\phi)
\]
By Lemma \ref{property}, it is enough to show
\[
M/(\cap_{j\neq i}Y_j+Y_i) \in \D^2_R
\]
Notice that $\Ass_R(M/Y_i)=\{P_i\}$ gives us $\sqrt{\Ann_{R}(M/Y_i)}=P_i$ by \cite[Proposition 8.B ]{matsumura1970commutative}. Since $M\in \D_R^1$, we have $M/Y_i\in \D_R^1$, which means $\Ht_R(\Ann_R(M/Y_i))\geq 2$. Hence $\Ht_R(P_i)\geq 2$.

Put $I_i:=\Ann_R(M/(\cap_{j\neq i}Y_j+Y_i))$. Since $\Ann_R(M/Y_i)\subset I_i$,
\[
P_i=\sqrt{\Ann_R(M/Y_i)}\subset \sqrt{I_i}
\]
which tells us $\Ht_R(I_i)\geq 2$. If we assume that $\Ht_R(I_i)=2$, then $P_i=\sqrt{I_i}$ and $\Ht_R(P_i)=2$. Hence, 
\[
\cap_{j\neq i}\Ann_R(M/Y_j)\subset I_i\subset P_i
\]
    By  \cite[Proposition 1.11]{atiyah2018introduction}, 
\[
\Ann_R(M/Y_j)\subset P_i
\]
for some $j\neq i$. We therefore have 
\[
P_j=\sqrt{\Ann_R(M/Y_j)}\subset P_i
\]
This tells us $P_j=P_i$ since $\Ht_R(P_i)=2$ and $\Ht_R(P_j)\geq 2$. This is a contradiction because $P_j\neq P_i$ by definition. Therefore, we have  $\Ht(I_i)\geq 3$ for $1\leq i\leq r$, which implies that $M/(\cap_{j\neq i}Y_j+Y_i) \in \D^2_R$.
\end{proof}

Since $D\in \D_R^2$, by Lemma \ref{property}, we have $D/TD\in \D_R^2$ and $D[T]\in \D_R^2$. By Lemma \ref{quotient}, we have $D/TD\in \D_{R/T}^1$ and $ D[T]\in \D_{R/T}^1$. We have the following exact sequence:
\[
D[T]\rightarrow M/TM\rightarrow\bigoplus_{i=1}^rM/(Y_i+TM)\rightarrow D/TD\rightarrow 0
\]
By Lemma \ref{property}, we know $M/T\in \D_{R/T}^1$ if and only if $M/(Y_i+TM)$ is in $ \D_{R/T}^1$ for all $1\leq i\leq r$.
\begin{lemma}
    With the same setup as above, $M/(Y_i+TM)\not\in \D_{R/T}^1$ if and only if $T\in P_i$ and $\Ht(P_i)=2$.
\end{lemma}
\begin{proof}
    Assume $M/(Y_i+TM)\not\in \D_{R/T}^1$. Then $M/Y_i$ is not in $\D_R^2$ by Lemma \ref{property} and Lemma \ref{quotient}. Since $M/Y_i$ is a finitely generated $R$-module, there exists a surjective map 
    \[
    (R/J_i)^n\twoheadrightarrow M/Y_i
    \]
    where $J_i=\Ann_R(M/Y_i)$. We have 
    \[
     (R/(J_i+TR))^n\twoheadrightarrow M/(Y_i+TM)
    \]
    Hence $R/(J_i+TR)\not\in\D_{R/T}^1$ by Lemma \ref{property}. 

    If $\Ht_R(J_i+TR)\geq 3$, then $R/(J_i+TR)\in \D_R^2$, which implies $R/(J_i+TR)\in \D_{R/T}^1 $ by Lemma \ref{quotient}. Contradiction! Therefore, $\Ht_R(J_i+TR)\leq 2$. 
    \[
    P_i\subset P_i+TR=\sqrt{J_i}+TR\subset \sqrt{J_i+TR}
    \]
    Since $\Ht_R(P_i)\geq 2$, we have $P_i=P_i+TR=\sqrt{J_i+TR}$ and $\Ht_R(P_i)=2$. Hence,  $T\in P_i$ and $\Ht(P_i)=2$.

    Conversely, assume $T\in P_i$ and $\Ht(P_i)=2$. Then $T^n\in \Ann_R(M/Y_i)$  for some integer $n$ since $P_i=\sqrt{\Ann_R(M/Y_i)}$. Assume $M/(Y_i+TM)\in \D_{R/T}^1$. Then $M/(Y_i+TM)\in \D_{R}^2$ by Lemma \ref{quotient}. We have the following exact sequence: 
    \[
    M/(Y_i+TM)\xrightarrow{\times T}M/(Y_i+T^2M)\rightarrow M/(Y_i+TM)\rightarrow 0
    \]
    By Lemma \ref{property}, we have $M/(Y_i+T^2M)\in \D_{R}^2$. Continuing the same process, we have $M/(Y_i+T^kM)\in \D_{R}^2$ for any integer $k\geq 1$. In particular, $M/Y_i=M/(Y_i+T^nM)\in \D_{R}^2$. This contradicts the fact that $\Ht_R(P_i)=2$.   
\end{proof}
\section{Characteristic ideal}

If $R$ is a Krull domain, we can define the characteristic ideal of a finitely generated $R$-module $M$. 
Recall that if $R$ is a Krull domain, then for any finitely generated torsion $R$-module $M$, there exists a short exact sequence:
\[
0\rightarrow A\rightarrow M\rightarrow \bigoplus_{i=1}^n R/\p_i^{e_i} \rightarrow B\rightarrow 0,
\]
where $A$ and $B$ are pseudo-null $R$-modules and $\p_i$ is a prime ideal of height 1. The $R$-module 
\[
E(M) := \bigoplus_{i=1}^n R/\p_i^{e_i}
\]
is called the elementary module attached to $M$. The characteristic ideal of $M$ is then defined as
\[
\ch_R(M) := \prod_{i=1}^n \p_i^{e_i}.
\]
If $M$ is a finitely generated $R$-module but not torsion, we set $\ch_R(M) := 0$. Hence, a finitely generated $R$-module $M$ is pseudo-null if and only if $\ch_R(M) = R$. If $M',M,M''$ are finitely generated $R$-modules and 
\[
0\rightarrow M'\rightarrow M\rightarrow M''\rightarrow 0,
\]
then we have $\ch_R(M)=\ch_R(M')\ch_R(M'')$.

There is another way to describe the characteristic ideal of a finitely generated torsion $R$-module $M$:
\[
\ch_R(M) = \prod_{\p} \p^{l_{\p}(M\otimes_R R_{\p})},
\]
where the product runs over all prime ideals $\p$ of height 1, and $l_{\p}(M\otimes_R R_{\p})$ denotes the length of the $R_{\p}$-module $M\otimes_R R_{\p}$. 

\begin{proposition}\label{char}
  Let $R$ and $R/T$ both be Noetherian Krull domains. Let $M$ be a finitely generated pseudo-null $R$-module. Then 
    \[
    \ch_{R/T}(M/T) = \ch_{R/T}(M[T]).
    \]
\end{proposition}

This result can be viewed as a generalization of \cite[Prop 2.7]{bandini2013characteristic}. They prove it only for $R$ of the form $R=A[[T]]$, where $A$ is a Noetherian Krull domain. Our method is different and uses Theorem \ref{ht=2}. Before proving the proposition, we first establish the following lemma. Throughout this section, we assume that $R$ and $R/T$ are both Noetherian Krull domains so that characteristic ideals are well-defined. Let $\pi: R \rightarrow R/T$ be the quotient map.

\begin{lemma}\label{ass}
    Let $N$ be a finitely generated pseudo-null $R$-module and assume that $\Ass_R(N) = \{P\}$. 
    
    If $T \in P$ and $\Ht(P) = 2$, then 
    \[
    \ch_{R/T}(N/T) = \ch_{R/T}(N[T]) = \pi(P)^{l_{P}((N/T) \otimes_R R_{P})}.
    \]
    
    Otherwise, $\ch_{R/T}(N/T) = \ch_{R/T}(N[T]) = R/T$.
\end{lemma}

\begin{proof}
    Let $\bar{Q}$ be a prime ideal of height 1 in $R/T$, and let $Q = \pi^{-1}(\bar{Q})$ be the preimage of $\bar{Q}$. Then $Q$ is a prime ideal in $R$ of height 2. Consider the short exact sequence:
    \[
    0 \rightarrow N[T] \rightarrow N \xrightarrow{\times T} N \rightarrow N/TN \rightarrow 0.
    \]
    Taking the localization at the prime ideal $Q$, we obtain:
    \[
    0 \rightarrow N[T] \otimes_R R_Q \rightarrow N \otimes_R R_Q \xrightarrow{\times T} N \otimes_R R_Q \rightarrow N/TN \otimes_R R_Q \rightarrow 0.
    \]
    We have $N \otimes_R R_Q \neq 0$ if and only if $\Ann_R(N) \subseteq Q$. Since $P = \sqrt{\Ann_R(N)} \subseteq Q$ and $\Ht_R(P) \geq 2$ while $\Ht_R(Q) = 2$, it follows that $\Ht_R(P) = 2$ and $T \in Q = P$. 

    Since $N$ is finitely generated, there exists an integer $n$ such that $Q^n \subseteq \Ann_R(N)$. Thus, $N \otimes_R R_Q$ can be viewed as a module over the Artinian ring $R_Q/Q^n$, which implies that the length $l_Q(N \otimes_R R_Q)$ is finite. Consequently, the first and last terms in the exact sequence have the same finite length as $R_Q$-modules. 

    The conclusion follows from the following observation: If an $R$-module $M$ is annihilated by $T$, then 
    \[
    M \otimes_R R_Q = M \otimes_{R/T} (R/T \otimes_R R_Q) = M \otimes_{R/T} (R/T)_{\bar{Q}}.
    \]
    Any $R_Q$-submodule of $M \otimes_R R_Q$ is also an $(R/T)_{\bar{Q}}$-module. Hence, the length of $M \otimes_R R_Q$ as an $R_Q$-module is the same as the length of $M \otimes_{R/T} (R/T)_{\bar{Q}}$ as an $(R/T)_{\bar{Q}}$-module.
\end{proof}

\begin{proof}[Proof of Proposition \ref{char}]
    Recall that we have the following exact sequence:
    \[
    0 \rightarrow M[T] \rightarrow \bigoplus_i (M/Y_i)[T] \rightarrow D[T] \rightarrow M/T \rightarrow \bigoplus_i M/(Y_i+TM) \rightarrow D/TD \rightarrow 0.
    \]

    We know that $D/TD$ and $D[T]$ are pseudo-null as $R/T$-modules. By Lemma \ref{ass}, we also have 
    \[
    \ch_{R/T}((M/Y_i)[T]) = \ch_{R/T}(M/(Y_i+TM)).
    \]
    Hence, it follows that 
    \[
    \ch_{R/T}(M[T]) = \ch_{R/T}(M/TM).
    \]
\end{proof}

 Let $R$ and $R/T$ both be Noetherian Krull domains. Let $M$ be a finitely generated torsion $R$-module.
Define \[
\Delta_{R/(T)}(M)= \frac{ \ch_{R/T}(M/TM)}{\ch_{R/T}(M[T]) } .
\]
 Here the quotient is understood formally, i.e. by subtracting the exponents at each height-one prime. By Proposition \ref{char}, we see that $\Delta_{R/(T)}(M)=1$ when $M$ is pseudo-null over $R$. Let $0\to M'\to M\to M''\to 0$ be an exact sequence of finitely generated torsion modules over $R$. Applying multiplication by $T$ and the snake lemma, we have 
\[
0\to M'[T]\to M[T]\to M''[T]\to M'/TM'\to M/TM\to M''/TM''\to 0.
\]
By the multiplicative property of the characteristic ideal, we have 
\[
\Delta_{R/(T)}(M)=\Delta_{R/(T)}(M')\Delta_{R/(T)}(M'').
\]

Put $S=R/TR$ and $\pi: R\to S$. Let $J$ be an ideal of $R$ such that $\pi(J)=(J+TR)/TR\neq 0$. We define 
\[
\pi^{div}(J):=\prod_{\Ht_S(\mathfrak{q})=1}\mathfrak{q}^{l_\mathfrak{q}(S_\mathfrak{q}/\pi(J)S_\mathfrak{q})},
\]
where $l_\mathfrak{q}(S_\mathfrak{q}/\pi(J)S_\mathfrak{q})$ is the length of $S_\mathfrak{q}/\pi(J)S_\mathfrak{q}$ as a $ S_\mathfrak{q}$-module. By definition, we have $\ch_S(S/\pi(J))=\pi^{div}(J)$. 
\begin{proposition}
    Let $R$ and $S:=R/(T)$ both be Noetherian Krull domains. Let $M$ be a finitely generated torsion $R$-module, and assume that $M/TM$ is torsion over $R/TR$. Then 
    \[
  \ch_{R/T}(M[T])  \pi^{div}(\ch_R(M)) = \ch_{R/T}(M/TM). 
    \]
 
\end{proposition}
The proposition can be viewed as a generalization of \cite[Prop 2.10 ]{bandini2013characteristic}, which proves the corresponding formula in the special case $R=S[[T]]$.
\begin{proof}
  Since $M/TM$ is an $R/T$-torsion module, we have $M_{(T)}=0$ by Nakayama's lemma and  $\operatorname{Char}_R(M) \nsubseteq (T)$. Consider $M$ together with its elementary module:
  \[
0\rightarrow A\rightarrow M\rightarrow E(M) :=\bigoplus_{i=1}^n R/\p_i^{e_i} \rightarrow B\rightarrow 0,
\]
where $A$ and $B$ are pseudo-null $R$-modules. Hence, $\Delta_{R/(T)}(M)=\Delta_{R/(T)}(E(M))= \prod_{i=1}^n \Delta_{R/(T)}(R/\p_i^{e_i})$. Take  a prime filtration of $R/\p_i^{e_i}$,
\[
0\subset N_0\subset N_1\subset N_2\subset \cdots\subset R/\p_i^{e_i}.
\]
Since each quotient factor $N_{i+1}/N_{i}$ is isomorphic to $R/\p_i$, we have $\Delta_{R/(T)}(R/\p_i^{e_i})=\Delta_{R/(T)}(R/\p_i)^{e_i} $. 

Since $M$ is an $S = R/T$-torsion module, we have $T \notin \mathfrak{p}_i$, and hence $(R/\mathfrak{p}_i)[T] = 0$. By definition,
$\operatorname{Char}_S(R/\mathfrak{p}_i + TR) = \pi^{\mathrm{div}}(\mathfrak{p}_i).$
Hence, $\Delta_{R/(T)}(E(M)) = \pi^{\mathrm{div}}(\operatorname{Char}_R(M))$.
   The conclusion follows by the definition of $\Delta_{R/(T)}$.
\end{proof}

\subsection{Application to number theory}
Here we mention a trick that is used in Iwasawa theory to illustrate the importance of Theorem \ref{ht=2}.

Let $f, g$ be two coprime polynomials in $\Lambda_d$. Then $\Lambda_d/(f,g)$ is pseudo-null over $\Lambda_d$. In the short exact sequence of finitely generated $\Lambda_d$-modules,  \[0\to A\to \Lambda_d/(f,g)\to B \to 0,\] the modules $A$ and $B$ are also pseudo-null $\Lambda_d$-modules. We may specialize to $\Lambda_{d-1}=\Lambda_d/(T_d)$. Let $\bar{f}, \bar{g}$ be the images of $f,g$ in $\Lambda_{d-1}$. We have the following exact sequence: 
\[
0\to D\to A/T_dA\to \Lambda_{d-1}/(\bar{f},\bar{g})\to B/T_dB\to 0,
\]
where $D=\im(B[T_d]\to A/T_dA)$. Hence,
\begin{equation}\label{eq char}
    \ch_{\Lambda_{d-1}}(A/T_dA)\ch_{\Lambda_{d-1}}( B/T_dB)=\ch_{\Lambda_{d-1}}(D)(\gcd(\bar{f},\bar{g})). 
\end{equation}
\begin{claim}
     $\gcd(\bar{f},\bar{g})$ is a unit if and only if $A/T_dA$ and $ B/T_dB$ are pseudo-null over $\Lambda_{d-1}$.
\end{claim}
\begin{proof}
    If $\gcd(\bar{f},\bar{g})$ is a unit, then $\Lambda_{d-1}/(\bar{f},\bar{g})$ is pseudo-null. Hence, $B/T_dB$ is pseudo-null. By Proposition \ref{char}, we know $\ch_{\Lambda_{d-1}}(B[T_d])=\ch_{\Lambda_{d-1}}(B/T_dB) $. Since $D$ is the image of $B[T_d]$, we have $\ch_{\Lambda_{d-1}}(D) \mid \ch_{\Lambda_{d-1}}(B/T_dB)$. Since $D$ is a subgroup of $A/T_dA$, we have $\ch_{\Lambda_{d-1}}(D) \mid \ch_{\Lambda_{d-1}}(A/T_dA)$. By equality \ref{eq char}, we know $\ch_{\Lambda_{d-1}}(D)=(1)$. Hence, $\ch_{\Lambda_{d-1}}(A/T_dA)=\ch_{\Lambda_{d-1}}(B/T_dB)=(1)$. Therefore, $A/T_dA$ and $B/T_dB$ are pseudo-null. 

    If $A/T_dA$ and $ B/T_dB$ are pseudo-null over $\Lambda_{d-1}$, then by the exact sequence, $\Lambda_{d-1}/(\bar{f},\bar{g})$ is pseudo-null. Hence, $\gcd(\bar{f},\bar{g})$ is a unit.
\end{proof}

In Iwasawa theory, people often need to consider whether $\gcd(\bar{f},\bar{g})$ is a unit. Now, the question is transferred to determining the pseudo-nullity of a quotient of a pseudo-null module, which is the central question addressed in our paper. We provide a few examples, omitting the technical details of Iwasawa theory to keep the focus clear.

Under the hypotheses of \cite[(1.6)]{Bleher2020HigherCC}
there is an exact sequence
\[
0 \to X_W^\psi \to \frac{\Lambda_W}{\mathcal{L}_{\mathfrak{p},\psi}\Lambda_W + \mathcal{L}_{\bar{\mathfrak{p}},\psi}\Lambda_W} \to \alpha(X_W^{\omega\psi^{-1}})(1) \to 0.
\]
Hence, if one specializes by quotienting a line generated by $T$, then the two specialized $p$-adic
$L$-functions  $\bar{\mathcal{L}}_{\mathfrak{p},\psi}$ and $\bar{\mathcal{L}}_{\bar{\mathfrak{p}},\psi}$ have a common height-one divisor if and only if at least
one of the two specialized end terms fails to be pseudo-null.
Related topics also appear in many papers \cite[Cor 4.21]{LimWangYan2026}, \cite{LeiPalvannan2022},\cite{LeiSujatha2021}. For example, \cite{LeiSujatha2021} considers the greatest common divisor $\gcd(L_p^+(E),L_p^-(E))$ of the plus and minus $p$-adic $L$-functions.

\section{Pseudo-nullity in non-commutative Iwasawa algebras}

In the previous section, we said that an $R$-module $M$ is pseudo-null if $M_\p=0$ for any prime ideal $\p$ of height 1. In a non-commutative Iwasawa algebra, there is another generalization of the definition of a pseudo-null module. We will recall the definition and its relation to the classical one. We will try to extend the results in the previous sections to the general case. 

Let $R$ be a ring and $M$ be an $R$-module. Let $E_R(M)$ be the injective envelope of $M$ \cite[Chap.V \S2]{stenstrom2012rings}. One can think of $E_R(M)$ as the ``minimal'' injective $R$-module containing $M$ in some sense. View $R$ as an $R$-module. Consider the ``minimal'' injective resolution of $R$: 
\[
0\rightarrow R \xrightarrow{\mu_0} E_0\xrightarrow{\mu_1} E_1\xrightarrow{\mu_2}\cdots
\]
In other words, $E_0=E_R(R)$ and $E_i=E_R(\coker\mu_{i-1})$. Define the full subcategory $\C^n_R$ of the category of $R$-modules consisting of all $R$-modules $M$ such that $\Hom_R(M,E_0\oplus E_1\oplus E_2\oplus\cdots \oplus E_n)=0$. The subcategory $\C^n_R$ is studied in \cite[Chap.VI]{stenstrom2012rings} and is related to Hereditary Torsion Theory. 

\begin{proposition}[Chap.VI, Prop. 6.9 in \cite{stenstrom2012rings}]\label{ext meaning}
    An $R$-module $M$ lies in $\C_R^n$ if and only if $\Ext^i_R(M',R)=0$ for any $i\leq n$ and any submodule $M'\subset M$.
\end{proposition}
We can relate $\C_R^n$ to $\D_R^n$ when $R $ is a Noetherian commutative ring. Let $R$ be a Noetherian commutative ring. We say that $R$ satisfies Serre's condition $(S_n)$ if 
\[
\depth(R_\p)\geq \min\{n,\Ht(\p)\}
\]
\begin{proposition}[Chap.VII, Prop 6.8 in \cite{stenstrom2012rings}]\label{Sn}
    If a Noetherian commutative ring $R$ satisfies Serre's condition $(S_{n+1})$, then an $R$-torsion module $M$ lies in $\C_R^n$ if and only if $M_\p=0$ for all primes $\p$ with $\Ht(\p)\leq n$.  
\end{proposition}
An integrally closed Noetherian commutative integral domain satisfies Serre's condition $(S_2)$. A Cohen–Macaulay ring satisfies Serre's condition $(S_n)$ for all $n$. Now we give the definition of pseudo-null modules that is often used in noncommutative Iwasawa theory. 

\begin{definition}\label{non def}
    Let $R$ be a ring. Let $M$ be a finitely generated right $R$-module. We call $M$ a pseudo-null $R$-module if $M\in \C_R^1$. 
\end{definition}

The proof of Theorem \ref{ht=2} depends on Lemma \ref{property} and Lemma \ref{quotient}. It is easy to see that the analogue of Lemma \ref{property} also holds for $\C_R^n$. 
\begin{lemma}\label{analog exact}
  Let $0 \rightarrow M' \rightarrow M\rightarrow M''\rightarrow 0$ be a short exact sequence. Then $M\in \C_R^n$ if and only if $M'\in \C_R^n$ and $M''\in \C_R^n$. 
\end{lemma}

Fix a central element $T\in Z(R)$ that is not a zero divisor in $R$.  
To prove the analogue of Lemma \ref{quotient}, we need the following lemma.
\begin{lemma}[Rees change of rings]\label{non quotient}
    Let $N$ be a right $R$-module annihilated by $T$. Then there are natural isomorphisms of left $R/TR$-modules,
    \[
    \Ext^i_{R/TR}(N,R/TR)\cong \Ext^{i+1}_R(N,R), 
    \] 
    for all $i\geq 0$.
\end{lemma}
\begin{proof}
    It is a special case of \cite[Lemma 1.2]{AjitabhSmithZhang}. Notice that we require that $T$ is a central element in $R$. Hence, the automorphism $\sigma$ in \cite[Lemma 1.2]{AjitabhSmithZhang} is trivial. 
\end{proof}

\begin{lemma}\label{analog quotient}
    Let $T$ be a central element in $R$ and assume $T$ is regular. Let $M$ be a right $R$-module. Then 
    \[
    M/MT\in \C_{R/TR}^n \quad \text{if and only if} \quad M/MT \in \C_R^{n+1}.
    \]
\end{lemma}

\begin{proof}
    Take $N$ to be a submodule of $M/MT$. Since $TN=0$ and $T$ is a nonzero divisor in $R$, we have $ \Hom_R(N,R)=0$. The conclusion follows from Lemma \ref{non quotient}.
\end{proof}

Now we give a generalization of Theorem \ref{ht=2} to a noncommutative ring. However, the criterion is not as explicit as that of Theorem \ref{ht=2}, which uses associated primes. It is also interesting to remark that the proof does not require the ring $R$ or the module $M$ to be Noetherian. Later, we illustrate the connection between the criterion and associated primes by restricting to a Noetherian commutative Cohen-Macaulay ring.

\begin{theorem}\label{non ht=2}
  Let $T$ be a central element in the ring $R$ and assume $T$ is regular. Let $M$ be a right $R$-module. If $M\in \C_R^n$, then $M/MT$ is \textbf{not} in $\C_{R/TR}^n$ if and only if there exists a submodule $M'$ with $MT\subset M'\subset M$ such that $\Ext^{n+1}_R(M',R)[T]\neq 0$.
\end{theorem}

Before we prove Theorem \ref{non ht=2}, we need to prepare some lemmas.

\begin{lemma}\label{non iso}
  Let $T$ be a central element in the ring $R$ and assume $T$ is regular. Let $M$ be a right $R$-module. If $M\in \C_R^n$, then there are isomorphisms
    \[
    \Ext_{R/TR}^i(M/MT,R/TR)\cong \Ext^i_R(M,R/TR)
    \]
    for $0\leq i\leq n$.
\end{lemma}

\begin{proof}
 We will use the Grothendieck spectral sequence. Consider the following two left-exact functors.
\[\begin{array}{ccccc}
 \{R\text{-modules }\}^{op} &\xrightarrow{F}  &  \{R/TR \text{-modules }\}^{op} &\xrightarrow{G}  & \{\text{ abelian groups }\}\\
 M &\rightarrow  & M\otimes_R R/TR  &  &\\
  &  & S & \rightarrow & \Hom_{R/TR}(S,R/TR)
 \end{array}\]
By Hom-tensor adjunction, we have 
\[\begin{array}{rl}
   G\circ F(M)  &= \Hom_{R/TR}(M\otimes_RR/TR,R/TR) \\
     & =\Hom_{R}(M,\Hom_{R/TR}(R/TR,R/TR))\\
     &=\Hom_R(M,R/TR)   
\end{array}
\]
Let $M$ be an injective module in $\{R \text{ module }\}^{op} $. Then $M$ is a projective $R$-module. Hence, $M/MT$ is a projective $R/TR$-module. Therefore, $M/T$ is $G$-acyclic since $\Ext_{R/TR}^i(M/T,R/TR)=0$ for $i\geq 1$. We have the spectral sequence:
\[
\Ext_{R/TR}^p(\Tor_q^R(M,R/TR),R/TR)\Rightarrow \Ext_R^{p+q}(M,R/TR).
\]
Since $0\to R\xrightarrow{\times T} R\to R/TR\to 0$, the projective dimension of $R/TR$ as an $R$-module is at most one. We have  
\[
\Tor_q^R(M,R/TR)=0,
\]
for $q\neq 0,1$. By \cite[Exercise~5.2.2]{Weibel} or \cite[Proposition~2.5]{FuAlgebraicGeometry}, we have a long exact sequence,
\[
\begin{array}{c}
    \cdots\rightarrow\Ext_{R/TR}^{p-2}(\Tor_1^R(M,R/TR),R/TR)
    \rightarrow \Ext_{R/TR}^p(\Tor_0^R(M,R/TR),R/TR)
    \rightarrow \Ext_R^p(M,R/TR)\\
    \rightarrow \Ext_{R/TR}^{p-1}(\Tor_1^R(M,R/TR),R/TR)
    \rightarrow \Ext_{R/TR}^{p+1}(\Tor_0^R(M,R/TR),R/TR)
    \rightarrow \Ext_R^{p+1}(M,R/TR)
    \rightarrow \cdots
\end{array}
\]
Since $M\in \C_R^n$, we have $\Tor_1^R(M,R/TR)\cong M[T]\in \C^n_R$. By Lemma \ref{analog quotient}, we have $\Tor_1^R(M,R/TR)\in \C^{n-1}_{R/TR}$. Thus, $\Ext_{R/TR}^{i}(\Tor_1^R(M,R/TR),R/TR)=0$ for $0\leq i\leq n-1$. By the long exact sequence and $\Tor_0^R(M,R/TR)=M/MT$, we have 
\[
\Ext_{R/TR}^i(M/MT,R/TR)
   \cong \Ext_R^i(M,R/TR)
\]
for $1\leq i\leq n$.   When $i=0$, we have $\Hom_{R/TR}(M/MT,R/TR)= \Hom_R(M,R/TR)$ by definition. 
\end{proof}

\begin{lemma}\label{non lift}
      Let $T$ be a central element in the ring $R$ and assume $T$ is regular. Let $M$ be a right $R$-module. If $M\in \C_R^n$, then there is an isomorphism
    \[
    \Ext_{R/TR}^n(M/MT,R/TR)\cong \Ext^{n+1}_R(M,R)[T].
    \]
\end{lemma}
\begin{proof}
    Apply $\Hom_R(M,-)$ to 
    \[
    0\to R\xrightarrow{\times T} R\to R/TR\to 0.
    \]
    We get a long exact sequence,
    \[
    \begin{array}{c}
    \cdots
    \rightarrow \Ext_{R}^n(M,R)
    \xrightarrow{\times T} \Ext_{R}^n(M,R)\\
    \rightarrow \Ext_{R}^n(M,R/TR)
    \rightarrow \Ext_{R}^{n+1}(M,R)
    \xrightarrow{\times T} \Ext^{n+1}_R(M,R)
    \rightarrow \cdots
\end{array}.
    \]
    Since $\Ext_{R}^n(M,R)=0$, Lemma \ref{non iso} gives 
    \[
     \Ext_{R/TR}^n(M/MT,R/TR)\cong\Ext_{R}^n(M,R/TR)=\Ext^{n+1}_R(M,R)[T]\]
\end{proof}

\begin{proof}[Proof of Theorem \ref{non ht=2}]
    Since $M\in \C_R^n$, we have $M/MT\in \C_R^n$ by Lemma \ref{analog exact}. By Lemma \ref{analog quotient}, we have $M/MT\in \C_{R/T}^{n-1}$. Hence, $M/MT\in \C_{R/T}^{n}$ if and only if $\Ext_{R/T}^n(N,R/T)=0$ for any submodule $N$ of $M/MT$. Let $M'$ be the preimage of $N$ inside $M$. Consider the following short exact sequence:
    \[
    0\to MT/M'T\to M'/M'T \to M'/MT\to 0.
    \]
    It gives a long exact sequence,
    \[
    \begin{array}{c}
    \cdots\rightarrow\Ext^{n-1}_{R/TR}(MT/M'T,R/TR)
    \rightarrow \Ext^{n}_{R/TR}(M'/MT,R/TR)\\
    \rightarrow \Ext^{n}_{R/TR}(M'/M'T,R/TR)\cong \Ext^{n+1}_{R}(M',R)[T]
    \rightarrow \Ext^{n}_{R/TR}(MT/M'T,R/TR)
    \rightarrow \cdots.
\end{array}
    \]
  The middle isomorphism is given by Lemma \ref{non lift}.

 If $M/MT\in \C_{R/T}^{n}$, since $M'/MT $ is a submodule of $M/MT$ and $MT/M'T$ is a quotient of the image of $M/MT$ under multiplication by $T$, we know $M'/MT\in \C_{R/T}^{n}$ and $MT/M'T\in \C_{R/T}^{n}$ by Lemma \ref{analog exact}. Hence, $ \Ext^{n}_{R/TR}(M'/MT,R/TR)=\Ext^{n}_{R/TR}(MT/M'T,R/TR)=0$. Therefore, $\Ext^{n+1}_{R}(M',R)[T]=0$ by the long exact sequence. 

If $M/MT\not \in \C_{R/T}^{n}$, then there exists a submodule $N=M'/MT$ such that $\Ext^{n}_{R/TR}(M'/MT,R/TR)\neq 0$. On the other hand, since 
$M/MT \in \C_{R/T}^{n-1}$ and $MT/M'T$ is a quotient of the image of $M/MT$ under multiplication by $T$, we have $\Ext^{n-1}_{R/TR}(MT/M'T,R/TR)=0$. Therefore, by the long exact sequence, $\Ext^{n}_{R/TR}(M'/MT,R/TR)\hookrightarrow \Ext^{n}_{R/TR}(M'/M'T,R/TR)$. Hence, $ \Ext^{n+1}_{R}(M',R)[T]\neq 0$.
\end{proof}

\subsection{Comparing with the commutative case}
Next, we show that Theorem \ref{non ht=2} can recover Theorem \ref{ht=2} when we require our ring $R$ to be Noetherian, commutative, and Cohen-Macaulay and the module $M$ to be finitely generated. When $R$ is Noetherian, commutative, and Cohen-Macaulay, by Proposition \ref{Sn}, Definition \ref{non def} of a pseudo-null module coincides with Definition \ref{def}.

We need to recall the properties of a minimal injective resolution for a Noetherian commutative Cohen-Macaulay ring $R$. Recall that the Bass number of the ring $R$ with respect to a prime $P\in \Spec(R)$ is defined as $\mu_i(P,R):=\dim_{k(P)}\Ext^i_{R_P}(k(P), R_P)$, where $k(P)$ is the residue field of $R$ at $P$. By \cite[3.2.9]{BrunsHerzog}, we have a minimal injective resolution for $R$:
\[
0\rightarrow R \xrightarrow{\mu_0} E_0\xrightarrow{\mu_1} E_1\xrightarrow{\mu_2}\cdots,
\]
where 
\[
E_i\cong \bigoplus_{P\in \Spec(R)} E_R(R/P)^{\mu_i(P,R)}.
\]
Since $R_P$ is Cohen--Macaulay, $\depth R_{P} = \dim R_{P} =
\operatorname{ht}(P)$.
By the definition of depth, we have $\Ext^i_{R_P}(k(P),R_{P})=0$ for $i<\operatorname{ht}(P)$, which implies $\mu_i(P,R)=0$ for $i<\operatorname{ht}(P)$. Hence, when $R$ is a Cohen-Macaulay ring, we can write
\[
E_i\cong \bigoplus_{\Ht(P)=i} E_R(R/P)^{\mu_i(P,R)}\oplus\bigoplus_{\Ht(P)<i} E_R(R/P)^{\mu_i(P,R)}.
\]

\begin{proposition}\label{ext ht}
  Let $R$ be a  Noetherian  commutative Cohen-Macaulay ring and let $M$ be a finitely generated pseudo-null $R$-module. Then
  \[
  \Ass_R(\Ext^2_R(M,R))=\{P\in \Ass_R(M): \Ht_R(P)=2\}.
  \]
\end{proposition}
  \begin{proof}
  
By \cite[\href{https://stacks.math.columbia.edu/tag/08Y8}{Lemma 08Y8}]{stacks-project}, we can view $E_R(R/P)$ as an $R_P$-module. If $\operatorname{ht}(P) \le 1$, then $M_P = 0$ since $M$ is pseudo-null over $R$. This means that for any $m \in M$, there exists $s \in R \setminus P$ such that $sm = 0$, which implies $\Hom_R(M, E_R(R/P))=0 $. Hence,
\[
\Hom_R(M, E_1)=0 .\]
By using the definition of the Ext functor in terms of an injective resolution, we have an injective map: 
\[
\Ext^2_R(M, R)\hookrightarrow  \Hom_R(M,E_2)=\Hom_R(M,  \bigoplus_{\operatorname{ht} P=2} E_R(R/P)^{\mu_2(P,R)})
\]
Since $M$ is a finitely generated $ R$-module, we have 
$R^m  \twoheadrightarrow M$ for some integer $m$. Hence, 
\[
\Hom_R(M,   E_R(R/P))\hookrightarrow \Hom_R(R^m,  E_R(R/P))\cong E_R(R/P)^m.\]
By the proof of \cite[Theorem~18.4 (v)]{MatsumuraCRT}, one has $\operatorname{Ass}_R E_R(R/P)=\{P\}.$  By \cite[Theorem~6.3]{MatsumuraCRT}, we have
\[
\Ass_R(\Ext^2_R(M, R))\subset \{P\in \Spec(R): \Ht(P)=2\}.
\]

Now, we begin the proof of the ``$\subset$'' direction. Let $P$ be an associated prime of $\Ext^2_R(M, R)$. Then $P\in \Supp_R( \Ext^2_R(M, R))$. We have 
\[
0\neq \Ext^2_R(M, R)_P\cong\Ext^2_{R_P}(M_P, R_P),
\]
which implies $M_P\neq 0$. Since $M$ is a pseudo-null module and $\Ht(P)=2$, $P$ is a minimal prime in $\Supp_R(M)$. Therefore, $P\in \Ass_R(M)$. 

For the ``$ \supset $'' direction, take any prime $ P\in \Ass_R(M)$ of height 2. Since $M$ is a pseudo-null $R$-module, we have $\Supp_{R_P}(M_P)=\{PR_P\}$. Then \cite[Prop 1.2.10(e)]{BrunsHerzog} tells us 
\[
\grade(PR_P,R_P)=\inf \{i:\Ext^i_R(M_P,R_P)\neq 0\}.
\]
By the definition of depth \cite[Def 1.2.7]{BrunsHerzog}, we have 
\[
\grade(PR_P,R_P)=\depth R_P=\dim R_P=2.
\]
Hence, 
\[
\Ext^2_R(M,R)_P\cong \Ext_{R_P}^2(M_P,R_P)\neq 0.
\]
Therefore, $P\in \Supp_R\Ext^2_R(M,R)$. Since $M$ is pseudo-null and $\Ht(P)=2$, the prime ideal $P$ is also minimal in $\Supp_R(\Ext^2_R(M,R))$. Hence, $P\in \Ass_R(\Ext^2_R(M,R))$.
  \end{proof}

\begin{corollary}\label{compare}
     Let $R$ be a Noetherian commutative  Cohen-Macaulay ring and let $M$ be a finitely generated pseudo-null $R$-module. Let $T$ be a regular nonunit element in $R$. Then $\Ext_R^2(M,R)[T]\neq 0$ if and only if there exists $P\in \Ass_R(M)$ with $\Ht(P)=2$ and $T\in P$.    

     Consequently, Theorem \ref{non ht=2} implies that $M/MT$ is not in $\C_{R/TR}^1$ if and only if there exists $P\in \Ass_R(M)$ with $\Ht(P)=2$ and $T\in P$.    
\end{corollary}
\begin{proof}
    The finitely generated $R$-module $\Ext_R^2(M,R)$ has nonzero $T$-torsion if and only if $T$ is a zero divisor on $\Ext_R^2(M,R)$, which holds if and only if $T$ is in some associated prime of $\Ext_R^2(M,R)$. By Proposition \ref{ext ht}, we get the equivalence. 
\end{proof}

\section{Application to number theory}\label{fine sel gp}
Let $K$ be a number field.
Let $\I$ be a commutative complete regular local ring of characteristic zero with a finite residue
field of characteristic $p$ and assume $p\geq 3$ and let $\mathbb T$ be a finite free
$\I$-module equipped with a continuous $\I$-linear action of
$\Gal(\bar{K}/K)$. Put
\[
  \I^\vee
  :=\Hom_{\Z_p,cts}(\I,\Q_p/\Z_p),
  \qquad
  \mathbb A
  :=\mathbb T\otimes_{\I}\I^\vee,
\]
where $\I^\vee$ is endowed with the trivial $\Gal(\bar{K}/K)$-action. Thus
$\mathbb A$ is a discrete $\I$-module with the induced $\Gal(\bar{K}/K)$-action.
 Let $S$ be a finite set of primes containing
the primes above $p$, the archimedean primes, and all primes at which
$\mathbb A$ is ramified, and write $G_{K,S}=\Gal(K_S/K)$,
where $K_S$ denotes the maximal algebraic extension of $K$ unramified
outside $S$. The fine Selmer group of $\mathbb A$ over $K$ is
defined by
\[
  R_S(\mathbb A/K)
  :=
  \ker\!\left(
    H^1(G_{K,S},\mathbb A)
    \longrightarrow
    \bigoplus_{v\in S} H^1(K_v,\mathbb A)
  \right),
\]
where the map is the product of the natural localization maps. All cohomology considered in the paper is continuous cohomology when it makes sense. 
More generally, if $L/K$ is an algebraic extension contained in $K_S$,
we define
\[
  R_S(\mathbb A/L)
  :=
  \ker\!\left(
    H^1(G_{L,S},\mathbb A)
    \longrightarrow
    \mathcal H_S^1(\mathbb A/L)
  \right),
\]
where $G_{L,S}:=\Gal(K_S/L)$ and
\[
  \mathcal H_S^1(\mathbb A/L)
  :=
  \varinjlim_{K\subseteq F\subseteq L}
  \bigoplus_{v\in S}\bigoplus_{w\mid v}
  H^1(F_w,\mathbb A),
\]
with the direct limit taken over the finite extensions $F/K$ contained
in $L$.
Its Pontryagin dual $Y_S(\mathbb A/L)
  :=
  R_S(\mathbb A/L)^\vee$
is called the dual fine Selmer group.

\begin{remark}\label{grade regular}
    Before we continue, we remark that pseudo-nullity behaves much better when we restrict to the number-theoretic setting. By \cite[Theorem~A.1(c)]{LimFineSelmer}, if $\I$ is a commutative complete regular local ring with a finite residue field of characteristic $p$ and $G$ is a compact pro $p$-adic analytic group without $p$-torsion, then $\I[[G]]$ is an Auslander regular ring.  Moreover, when \(\I\) has characteristic zero,
\cite[Proposition~2.6]{VenjakobCharacteristic} implies that
\(\I[[G]]\) has no zero divisors. Hence \(\I[[G]]\) is an Auslander
regular domain.
    
    We refer the reader to \cite{Venjakob2002} for the definitions and properties of the Auslander condition, Auslander Gorenstein rings, and Auslander regular rings. We note that, as in \cite[Remark 3.4]{Venjakob2002}, for a commutative ring $R$, Auslander Gorenstein is equivalent to Gorenstein in the usual sense, and Auslander regular is equivalent to regular with finite Krull dimension in the usual sense. 

One advantage of an Auslander Gorenstein ring is that we can control $\Ext_R^i(M', R)$ at the same time for all submodules $M'\subset M$. For a nonzero finitely generated module $M$ over a ring $R$, define its grade as 
\[
j_R(M):=\min \{i\geq 0:\Ext^i_R(M,R)\neq 0\},
\]
with $j_R(0)=\infty$. If $R$ is Auslander Gorenstein, then for every short exact sequence $0 \rightarrow M' \rightarrow M\rightarrow M''\rightarrow 0$ of finitely generated modules, $j_R(M)=\min\{j_R(M'), j_R(M'')\}$; see \cite[Prop 3.6(i)]{Venjakob2002}. Hence, for a finitely generated right module $M$ over an Auslander Gorenstein ring $R$, we have that $M\in \C_R^1$ if and only if $\Ext^0_R(M,R)=\Ext^1_R(M,R)=0$ if and only if $j_R(M)\geq 2$. Theorem \ref{non ht=2} then tells us that if $M$ is a finitely generated pseudo-null module over an Auslander Gorenstein ring $R$, then $M/MT$ is pseudo-null if and only if $\Ext^{2}_R(M,R)[T]= 0$.
\end{remark}

A \emph{specialization} of $\I$ is a local ring homomorphism
$  \phi:\I\longrightarrow\mathcal O_\phi$, where $\mathcal O_\phi$ is a commutative complete regular local
ring of characteristic zero and is finitely generated as an $\I$-module via $\phi$. We write $\PP_\phi:=\ker(\phi)$. Then $\PP_\phi$ is a prime
ideal of $\I$, and $\phi$ factors as
$ \I\twoheadrightarrow\I/\PP_\phi
    \hookrightarrow\mathcal O_\phi$, 
where the second map is a finite injective local homomorphism.

We write $\mathbb T_\phi
    :=
    \mathbb T\otimes_{\I}\mathcal O_\phi$
and
$ \mathbb A_\phi
    :=
    \mathbb T_\phi\otimes_{\mathcal O_\phi}
    \mathcal O_\phi^\vee
$, where $
    \mathcal O_\phi^\vee
    :=
    \Hom_{\Z_p,\mathrm{cts}}
    (\mathcal O_\phi,\Q_p/\Z_p).
$
Then there is a natural
$\Gal(\overline K/K)$-equivariant identification
\[
\begin{aligned}
\Hom_\I(\mathcal O_\phi,\A)
&\cong
\mathbb T\otimes_\I
\Hom_\I(\mathcal O_\phi,\I^\vee)\\
&\cong
\mathbb T\otimes_\I\mathcal O_\phi^\vee\\
&\cong
\mathbb{T}_\phi\otimes_{\mathcal O_\phi}\mathcal O_\phi^\vee
=\A_\phi.
\end{aligned}
\]
\begin{example}[Hida families, Galois representations and specializations]
For a classical Hecke eigenform $f\in S_k(\Gamma_0(N),\chi)$, we can associate a Galois representation with $f$:
\[
\rho_f: \Gal(\bar{\Q}/\Q)\to GL_2(\OO),
\]
where $\OO$ is the ring of integers of a local field extension over $\Q_p$ containing the coefficients of $f$.

Let $\Gamma\cong 1+p\Z_p$ and $\Lambda=\OO[[\Gamma]]$. An arithmetic point is a specialization $\xi: \Lambda\to \bar{\Q}_p$ such that $\xi(x)=\chi(x)x^k$, where $x\in \Gamma$, $k\in \Z$, and $\chi$ is a finite-order character $ \chi: \Gamma\to \bar{\Q}_p^\times$. Let $\mathcal{L}$ be a finite extension of the quotient field of $\Lambda$ and let $\mathbb{I}$ be the normalization of $\Lambda$ in $\mathcal{L}$. Then, we know $\mathbb{I}$ is a finite flat domain over $\Lambda$. We say $\xi: \mathbb{I}\to \bar{\Q}_p $ is an arithmetic point if $\xi|_{\Lambda}$ is an arithmetic point. A formal $q$-expansion $\mathcal{F}=\sum_{n=1}^\infty A_nq^n\in \I[[q]]$ is called a $\Lambda$-adic modular form of tame character $\chi$ and level $N$ if, for almost all arithmetic specializations $\xi$, the specialization $\xi(\mathcal{F})=\sum_{n=1}^\infty\xi(A_n)q^n$ is a $q$-expansion of a classical normalised ordinary eigenform of weight $k$. We can also associate $F$ with a Galois representation: 
\[
\rho_\mathcal{F}: \Gal(\bar{\Q}/\Q)\to GL_2(\I),
\]
and the specialization $\xi\circ \rho_\mathcal{F}$ is the Galois representation associated with the modular form $f=\xi(\mathcal{F})$. Writing $\PP=\ker\xi$, we have $\rho_f=\rho_\mathcal{F}\mod{\PP}$, where $f=\mathcal{F}\mod{\PP}$.
\end{example}

   Let $ L/K$ be a Galois extension inside $G_{K,S}$ containing the cyclotomic $\Z_p$-extension $K^{cyc}/K$ such that $G:=\Gal(L/K)$ is a compact pro $p$-adic Lie group without $p$-torsion. Denote $H:=\Gal(L/K^{cyc})$.

The generalized version of the Coates--Sujatha Conjecture A
\cite[Conjecture~A]{CoatesSujatha2005}, formulated for Galois representations over complete Noetherian coefficient rings in
\cite[Conjecture~A]{LimFineSelmer}, predicts that the dual fine Selmer group $Y(\mathbb{A}/K^{\mathrm{cyc}})$ is finitely generated over \(\I\).
Moreover, \cite[Lemma~5.2]{LimFineSelmer} shows that \(Y(\mathbb{A}/K^{\mathrm{cyc}})\) is finitely generated over \(\I\)
if and only if \(Y(\mathbb{A}/L)\) is finitely generated over
\(\I[[H]]\).

\begin{remark}\label{skew}
By \cite[Section 4]{SchneiderVenjakob2006}, $\I[[G]]$ is a skew power series ring over $\I[[H]]$. The advantage of this fact is that \cite[Prop 3.2]{SchneiderVenjakob2006} tells us that if $S$ is a skew power series ring over $R$ and the $S$-module $M$ is finitely generated over $R$, then $j_R(M)=j_S(M)-1$.  Hence, the module $M$ is pseudo-null over $S$ if and only if $M$ is torsion over $R$. 
\end{remark}

Now, we state a result of \cite{Jha2012} about specialization of the dual fine Selmer group along a Hida family. Denote the dual fine Selmer group associated to the representation of $\mathcal{F}$ as $Y(\mathcal{F}/L)$ and  the dual fine Selmer group associated to the representation of $f$ as  $Y(f/L)$.  
Under the assumptions in his paper \cite{Jha2012} that $K/\Q$ is abelian and that (Nor), (Irr), and (DimS) hold, he shows that if $Y(\mathcal{F}/L)$ is pseudo-null over $\I[[G]]$ and $Y(\mathcal{F}/L)$ is a finitely generated $\I[[H]]$-module, then the Iwasawa module $Y(f_\xi/L)$ is pseudo-null over $\OO_{f_\xi}[[G]]$ for all but finitely many arithmetic specializations $\xi$. His proof carefully calculates the $\I[[H]]$-rank of $Y(f_\xi/L)$.

 In Theorem \ref{generalize Jha}, we will give a generalization of Jha's result without the assumptions that $K/\Q$ is abelian and that (Nor), (Irr), and (DimS) hold, and we consider the big Galois representation $\mathbb{T}$ over any commutative complete regular local ring of characteristic zero with a finite residue field of characteristic \(p\). 

Notice that under the assumption $\I=\OO[[W]]$ in \cite{Jha2012}, the kernel of the specialization $\xi$ is principal. Our Theorem \ref{non ht=2} applies in this situation and gives an explicit description of the exceptional locus. We will discuss it in Subsection \ref{back to Jha}.

\begin{theorem}\label{generalize Jha}
    Let $ L/K$ be a Galois extension inside $G_{K,S}$ containing the cyclotomic $\Z_p$-extension $K^{cyc}/K$ such that $G:=\Gal(L/K)$ is a compact pro $p$-adic Lie group without $p$-torsion. Denote $H:=\Gal(L/K^{cyc})$.

    Assume the Iwasawa module of the dual fine Selmer group $Y_S(\mathbb A/L)$ is pseudo-null over $\I[[G]]$ and is finitely generated over $\I[[H]]$. Then there exists a Zariski dense open subset $U\subset \Spec(\I)$ such that for any specialization $\phi$ with $\PP_\phi=\ker\phi\in U $, the Iwasawa module of the dual fine Selmer group $Y_S(\mathbb{A}_\phi/L)$ is pseudo-null over $\OO_\phi[[G]]$. 
\end{theorem}
For simplicity, we say a property holds for Zariski generic specializations $\phi$ if there exists a Zariski open subset $U\subset \Spec(\I)$ such that the property holds for every $\phi$ with $\ker\phi\in U$.

By definition, we have the following diagram, and each row is exact.
\begin{equation}
    \label{compare terms}
 \begin{tikzcd}[column sep=tiny, row sep=small]
0 \arrow[r]
& R_S(\A_\phi/L)
  \arrow[r]
  \arrow[d,"r"]
& H^1(G_{L,S},\A_\phi)
  \arrow[r]
  \arrow[d,"\alpha"]
& \mathcal H_S^1(\A_\phi/L)
  \arrow[d,"\beta"]
\\
0 \arrow[r]
& \Hom_{\I}\!\left(
      \mathcal O_\phi,
      R_S(\A/L)
  \right)
  \arrow[r]
  \arrow[d]
& \Hom_{\I}\!\left(
      \mathcal O_\phi,
      H^1(G_{L,S},\A)
  \right)
  \arrow[r]
  \arrow[d]
& \Hom_{\I}\!\left(
      \mathcal O_\phi,
      \mathcal H_S^1(\A/L)
  \right)
  \arrow[d]
\\
0 \arrow[r]
& \Hom_{\I}\!\left(
      \I,
      R_S(\A/L)
  \right)
  \arrow[r]
  \arrow[d,"\sim"]
& \Hom_{\I}\!\left(
      \I,
      H^1(G_{L,S},\A)
  \right)
  \arrow[r]
  \arrow[d,"\sim"]
& \Hom_{\I}\!\left(
      \I,
      \mathcal H_S^1(\A/L)
  \right)
  \arrow[d,"\sim"]
\\
0 \arrow[r]
& R_S(\A/L) \arrow[r]
& H^1(G_{L,S},\A) \arrow[r]
& \mathcal H_S^1(\A/L).
\end{tikzcd}
\end{equation}
The vertical map in the first row is induced by evaluating the cohomology class at $\OO_\phi$ since $\A_\phi\cong\Hom_\I(\OO_\phi,\A)$. The vertical map in the second row is induced by the specialization $\phi:\I\to \OO_\phi$.

Hence, the proof is divided into two steps. First, we build a generic control theorem to compare $Y_S(\mathbb{A}_\phi/L)$ with $Y_S(\mathbb A/L)\otimes_\I\OO_\phi$. Secondly, we build a generic quotient theorem to compare  $Y_S(\mathbb A/L)\otimes_\I\OO_\phi$ with  $Y_S(\mathbb A/L)$.

\begin{theorem}\label{generic control}
        Let $ L/K$ be a Galois extension inside $G_{K,S}$ containing the cyclotomic $\Z_p$-extension $K^{cyc}/K$ such that $G:=\Gal(L/K)$ is a compact pro $p$-adic Lie group without $p$-torsion. Denote $H:=\Gal(L/K^{cyc})$.

        The kernel and cokernel of the natural map  $ Y_S(\mathbb A/L)\otimes_\I\OO_\phi\to Y_S(\mathbb{A}_\phi/L)$ are pseudo-null $\OO_\phi[[G]]$-modules for Zariski generic many specializations $\phi$. 
\end{theorem}
\begin{proof}
We have $Y_S(\mathbb{A}_\phi/L)= R_S(\A_\phi/L)^\vee$ and $ Y_S(\mathbb A/L)\otimes_\I\OO_\phi= \Hom_{\I}( \mathcal O_\phi,  R_S(\A/L))^\vee$.
By the commutative diagram \eqref{compare terms}, it is equivalent to showing that $\ker r$ and $\coker r$ are pseudo-null $\OO_\phi[[G]] $-modules for Zariski generic $\phi$. By the Snake lemma, it is enough for us to show that $\ker\alpha$, $\coker\alpha$, $\ker\beta$ are pseudo-null $\I[[G]]$-modules for Zariski generic $\phi$.

By Lemma \ref{five exact}, we know $\ker\alpha=\Ext_{\I}^{1}(\OO_\phi,\A^{G_{L,S}})$ and $\coker\alpha\subset \Ext_{\I}^{2}(\OO_\phi,\A^{G_{L,S}})$. By Lemma \ref{generic torsion}, we know the duals of both modules are finitely generated torsion $\OO_\phi$-modules for Zariski generic $\phi$, hence are pseudo-null $\OO_\phi[[G]]$-modules for Zariski generic $\phi$ by Remark \ref{skew}.

Fix a finite prime \(v\in S\) of \(K\). Let $\eta\mid v$ be a prime of \(K^{\mathrm{cyc}}\), and choose a prime $w\mid \eta$ of \(L\). Since \(H\) fixes \(K^{\mathrm{cyc}}\), the group \(H\)
acts  on the set of primes of \(L\) lying above \(\eta\). Define the
decomposition subgroup $H_w=\{h\in H:h(w)=w\}$.
By Lemma \ref{five exact}, $E_w:= \Ext_{\I}^{1}(\OO_\phi,\A^{G_{L_w}})$ is the kernel of $H^1(L_w,A_\phi)\to \Hom_\I(\OO_\phi,H^1(L_w,\mathbb A))$, whose dual is a finitely generated torsion $\OO_\phi$-module for Zariski generic $\phi$.

For \(h\in H\), conjugation identifies the local Galois groups at
\(w\) and \(hw\), and hence induces an isomorphism
$E_w\xrightarrow{\;\sim\;}E_{hw}$.
Thus the collection of $E_{w'}$ for all primes \(w'\mid \eta\) is obtained from \(E_w\) by allowing \(H\) to
act on the cosets \(H/H_w\). Equivalently, on the discrete side,
\[
E_\eta
\simeq
\operatorname{Coind}_{H_w}^{H}(E_w).
\]

Passing to Pontryagin duals converts continuous coinduction into
completed induction. We have 
$E_\eta^\vee
\simeq
\mathcal O_\phi[[H]]
\widehat\otimes_{\mathcal O_\phi[[H_w]]}
E_w^\vee.$
It is a finitely generated torsion $\OO_\phi[[H]] $-module for Zariski generic $\phi$ and hence is pseudo-null over $\OO_\phi[[G]]$ for Zariski generic $\phi$.

Finally, we have
\[
(\ker\beta)^\vee=\bigoplus_{v\in S}\bigoplus_{\eta\mid v}
\mathcal O_\phi[[H]]
\widehat\otimes_{\mathcal O_\phi[[H_w]]}
E_w^\vee,
\]
where for each \(\eta\mid v\) one chooses a prime \(w\mid\eta\) of
\(L\). The index set of the sum is finite, as every prime is finitely decomposed in the cyclotomic $\Z_p$-extension. Hence, $\ker\beta$ is pseudo-null over $\OO_\phi[[G]]$ for Zariski generic $\phi$.
  
\end{proof}
\begin{remark}\label{jha comment}
    If we assume $(\mathrm{Dim}_S)$ as \cite{Jha2012} that the decomposition group
$G_v$ satisfies $\dim G_v\geq 2$ for every $v\in S$, then $\dim H_w\geq 1 $. Hence, $E_w^\vee$ is a finitely generated torsion $\OO_\phi[[H_w]]$-module. Therefore, $(\ker\beta)^\vee$ is a pseudo-null $\OO_\phi[[G]]$-module automatically. Hence, assume $(\mathrm{Dim}_S)$, the natural map  $ Y_S(\mathbb A/L)\otimes_\I\OO_\phi\to Y_S(\mathbb{A}_\phi/L)$ is pseudo-isomorphism for every specialization $\phi$. Without assuming $(\mathrm{Dim}_S)$, even if we know $Y_S(\mathbb{A}_{\phi_0}/L)$ is pseudo-null  $\OO_{\phi_0}[[G]]$-module,  the natural map  $ Y_S(\mathbb A/L)\otimes_\I\OO_{\phi_0}\to Y_S(\mathbb{A}_{\phi_0}/L)$ may fail to be pseudo-isomorphism.  For instance, via the Shapiro
identification, one may take the anticyclotomic deformation
$\mathbb T=\Z_p[[\Gal(F_{\mathrm{ac}}/F)]]$ arising in
\cite[Proposition~6.1]{qi2026pseudonullityfineselmergroups}
to obtain such a counterexample.
\end{remark}

\begin{remark}
    The proof of Theorem \ref{generic control} does not use the assumption that the Iwasawa module of the dual fine Selmer group $Y_S(\mathbb A/L)$ is pseudo-null over $\I[[G]]$ and is finitely generated over $\I[[H]]$.
\end{remark}

Now, we prove the two lemmas that we used in the proof.
\begin{lemma}\label{five exact}
    Let $\I$ be a commutative complete regular local ring of characteristic zero with a finite residue field of characteristic $p$, and let $\mathbb T$ be a finite free $\I$-module equipped with a continuous $\I$-linear action of a profinite group
$\G$. Let $\phi$ be a specialization of $\I$ and put $ \mathbb A
  :=\mathbb T\otimes_{\I}\Hom_{\Z_p,cts}(\I,\Q_p/\Z_p)$ and $\A_\phi:=\Hom_\I(\OO_\phi,\A)$. We have
\[
\begin{aligned}
  0\longrightarrow\Ext_{\I}^{1}(\OO_\phi,\A^\G)
  &\longrightarrow H^1(\G,\A_\phi)
    \xrightarrow{\alpha_\G}\Hom_\I(\OO_\phi,H^1(\G,\A))\\
  &\longrightarrow\Ext_{\I}^{2}(\OO_\phi,\A^\G)
  \longrightarrow H^2(\G,\A_\phi).
\end{aligned}
\]
\end{lemma}
\begin{proof}
     We will use the Grothendieck spectral sequence. Consider the following two left-exact functors.
\[
\begin{array}{ccccc}
\{\I[[\G]]\text{-modules}\}
& \xrightarrow{\ F\ } &
\{\I\text{-modules}\}
& \xrightarrow{\ E\ } &
\{\OO_\phi\text{-modules}\}
\\[2mm]
M
& \longmapsto &
M^\G
&&
\\[2mm]
&&
N
& \longmapsto &
\Hom_{\I}(\OO_\phi,N).
\end{array}
\]
To get a Grothendieck spectral sequence
\[
 E_2^{a,b}
  =\Ext_{\I}^{a}\bigl(\OO_\phi,H^b(\G,\A)\bigr)
  \Longrightarrow R^{a+b}(E\circ F)(\A),
\]
we need to show that  $F$ must
send injective $\I[[\G]]$-modules to $E$-acyclic objects. Take an injective $\I[[\G]]$-module $M$.  Let $J$ be the kernel of the map $\I[[\G]]\to \I$; we have $M^\G=M[J]$. By \cite[ \S 3, Exercise 28(1)]{lam2012lectures},
$M[J]$ is an injective $\I[\G]/J\simeq\I$-module.

The associated five-term exact sequence in low degrees is
\begin{equation}\label{first exact five}
 \begin{aligned}
  0\longrightarrow\Ext_{\I}^{1}(\OO_\phi,\A^\G)
  &\longrightarrow R^1(E\circ F)(\A)
   \longrightarrow \Hom_\I(\OO_\phi,H^1(\G,\A))\\
  &\longrightarrow\Ext_{\I}^{2}(\OO_\phi,\A^\G)
   \longrightarrow R^2(E\circ F)(\A).
\end{aligned}   
\end{equation}

To compute $R^i(E\circ F)(\A)$, we use another Grothendieck spectral sequence. Consider the following two left-exact functors. 
\[
\begin{array}{ccccc}
\{\I[[\G]]\text{-modules}\}
& \xrightarrow{\ E_1\ } &
\{\OO_\phi[[\G]]\text{-modules}\}
& \xrightarrow{\ F_1\ } &
\{\OO_\phi\text{-modules}\}
\\[2mm]
M
& \longmapsto &
\Hom_\I(\OO_\phi,M)
\\[2mm]
&&
N
& \longmapsto &
N^\G.
\end{array}
\]
   Notice that the functor $E_1$ is right adjoint to the restriction of scalars functor $\Res$ from the category of $\OO_\phi[[\G]]$-modules to the category of $\I[[\G]]$-modules. In other words, we have $\Hom_{\OO_\phi[[\G]]}(N, \Hom_\I(\OO_\phi,M))\cong \Hom_{\I[[\G]]}(\Res N,M)$. It is known that the right adjoint of an exact functor preserves injectives. Hence, the functor $E_1$ maps injectives to injectives, which are of course $F_1$-acyclic.

Notice that $E_1(M)=\Hom_\I(\OO_\phi,M)=\Hom_{\I[[\G]]}(\OO_\phi[[\G]],M)$. Hence $R^iE_1(M)=\Ext^i_{\I[[\G]]}(\OO_\phi[[\G]],M)$. Let $P_\cdot$ be a finite free resolution of $\mathcal{O}_\phi$ as an $\I$-module. Then $P_\cdot \otimes_\I \I[[\G]]$ is a finite free resolution of $\mathcal{O}_\phi \otimes_\I \I[[\G]] = \mathcal{O}_\phi[[\G]]$ as an $\I[[\G]]$-module since $\I[[\G]]$ is flat over $\I$. Hence, $R^iE_1(M)=H^i(\Hom_{\I[[\G]]}(P_\cdot \otimes_\I \I[[\G]], M))=H^i(\Hom_\I(P_\cdot,M))=\Ext^i_\I(\OO_\phi,M) $.
We get
\[
  {}'E_2^{a,b}
  =H^a\bigl(\G,\Ext_{\I}^{b}(\OO_\phi,\A)\bigr)
  \Longrightarrow R^{a+b}(F_1\circ E_1)(\A).
\]
One can check $\I^{\vee}$ is an injective
$\I$-module by Baer's criterion. Since $\mathbb{T}$ is a finite free $\I$-module, $\A=\mathbb{T}\otimes_\I \I^\vee$ is injective over $\I$. Therefore, $\Ext_{\I}^{b}(\OO_\phi,\A)=0$ for $b\geq 1$. The spectral sequence degenerates and gives us
\[
H^n(\G,  \Hom_{\I}(\OO_\phi,\A) )\cong R^n(F_1\circ E_1)(\A).
\]
Notice that $E\circ F(M)=\Hom_{\I}(\OO_\phi,M^\G)=\Hom_{\I}(\OO_\phi,M)^\G=F_1\circ E_1(M)$. Hence, $R^n(F_1\circ E_1)(\A)=R^n(E\circ F)(\A)$. Combined with \eqref{first exact five}, we get the conclusion. 
\end{proof}
\begin{lemma}\label{generic torsion}
    With the same setup as in Lemma \ref{five exact}, for $i\geq 1$, the dual module $\Ext_{\I}^{i}(\OO_\phi,\A^\G)^\vee $ is a finitely generated torsion  $\OO_\phi$-module for Zariski generic specializations $\phi$ .
\end{lemma}
\begin{proof}
By \cite[\S 3]{Fieldhouse1975}, we have $\Ext_{\I}^{i}(\OO_\phi,\A^\G)^\vee\cong \Tor^\I_i(\OO_\phi, (\A^\G)^\vee)$. Both modules in the Tor groups are finitely generated over $\I$, hence $ \Tor^\I_i(\OO_\phi, (\A^\G)^\vee)$ is finitely generated over $\I$ and therefore is finitely generated over $\OO_\phi$.

Since $\I$ is a domain and $ (\A^\G)^\vee$ is a finitely generated $\I$-module, by \cite[\href{https://stacks.math.columbia.edu/tag/051S}{Lemma 051S}]{stacks-project}, there exists an element $a\in \I$ such that the localization  $((\A^\G)^\vee)_a$ is free over $\I_a$.  Take a specialization $\phi$ such that $\phi(a)\neq 0$. 
Let $K_\phi$ be the fraction field of $\OO_\phi$. We have a natural map $\I_a\to K_\phi $ as $\phi(a)\neq 0$.

Let $P_\cdot\to (\A^\G)^\vee $ be a finitely generated free resolution of $ (\A^\G)^\vee$ as an $\I$-module. We have
\begin{align*}
\Tor^\I_i(\OO_\phi, (\A^\G)^\vee)\otimes_{\OO_\phi} K_\phi =H_i(P_\cdot\otimes_\I \OO_\phi) \otimes_{\OO_\phi} K_\phi
  &\cong H_i(P_\cdot\otimes_\I \OO_\phi \otimes_{\OO_\phi} K_\phi) \\
  &= H_i(P_\cdot\otimes_\I \I_a\otimes_{\I_a} K_\phi) \\
  &= H_i((P_\cdot)_a\otimes_{\I_a} K_\phi) \\
  &\cong \Tor^{\I_a}_i( K_\phi, ((\A^\G)^\vee)_a)=0
\end{align*}
The first isomorphism is because $K_\phi$ is flat over $\OO_\phi$, hence tensoring with it does not change the exactness. The second isomorphism is because $(P_\cdot)_a $ is a free resolution of $((\A^\G)^\vee)_a$ as an $\I_a$-module since localization preserves exactness. The last equality holds because $((\A^\G)^\vee)_a$ is free as an $\I_a$-module. Therefore, 
$\Ext_{\I}^{i}(\OO_\phi,\A^\G)^\vee$ is a torsion $\OO_\phi$-module when $\phi(a)\neq 0$. Such $\phi$ are Zariski generic.

\end{proof}
Next, we deduce a generic quotient theorem to compare  $Y_S(\mathbb A/L)\otimes_\I \OO_\phi$ with  $Y_S(\mathbb A/L)$.

\begin{theorem}\label{generic quotient}
      Assume $\I$ is a commutative complete regular local ring of characteristic zero with a finite residue field of characteristic $p$ and $G$ is a compact pro $p$-adic analytic group without $p$-torsion. Let $H$ be a closed normal subgroup of $G$ such that $G/H\cong \Z_p$. 

      Let $M$ be a pseudo-null $\I[[G]]$-module and assume $M$ is finitely generated over $\I[[H]]$. Then $M_\phi:=M\otimes_\I \OO_\phi$ is a pseudo-null $\OO_\phi[[G]]$-module for Zariski generic specializations $\phi$.
\end{theorem}

\begin{proof}

By Remark \ref{grade regular}, $M$ is a pseudo-null $\I[[G]]$-module if and only if $j_{\I[[G]]}(M)\geq 2$. By \cite[Prop 3.2]{SchneiderVenjakob2006}, we have $j_{\I[[H]]}(M)=j_{\I[[G]]}(M)-1\geq 1 $. Conversely, to show $M_\phi$ is a pseudo-null $\OO_\phi[[G]]$-module, we only need to show $j_{\OO_\phi[[H]]}(M_\phi)\geq 1$. In other words, we only need to show $\Hom_{\OO_\phi[[H]]}(M_\phi,\OO_\phi[[H]])=0$.

By \cite[Prop 3.6(i)]{Venjakob2002}, we have $ j_{\I[[H]]}(M')\geq j_{\I[[H]]}(M)\geq 1$ for any submodule $M'\subset M$. 
 Let $m_1,\cdots, m_s$ be generators of $M$ as an $\I[[H]]$-module. There exists $s_i\in \I[[H]]$ for each $m_i$ such that $s_im_i=0$.  Otherwise, we have an isomorphism $ \I[[H]]\to\I[[H]]m_i$ by sending $b\to bm_i$. This contradicts the fact that $j_{\I[[H]]}(\I[[H]]m_i)\geq 1$ since  $\I[[H]]$ is a submodule of $M$.

By definition, we have $\I[[H]]=\varprojlim \I[H/U] $ where $U$ is open in $H$. Take $U$ such that $\pi_U(s_i)\neq 0$ for all $s_i$, where $\pi_U: \I[[H]]\to \I[H/U]$. Write $\pi_u(s_i)=\sum_{\bar{h}\in H/U} c_{i,\bar{h}} \bar{h}$. There is at least one nonzero coefficient $c_{i,\bar{h}}$, denoted by $c_i$. Put $a=\prod_i c_i$. Take a specialization $\phi$ such that $\phi(a)\neq 0$. Then the image of $s_i$ in $\OO_\phi[[H]] $ is nonzero. 

Take $f\in \Hom_{\OO_\phi[[H]]}(M_\phi,\OO_\phi[[H]])$. By Remark \ref{grade regular}, we also know $\OO_\phi[[H]]$ is a domain.  Notice that $m_1,\cdots, m_s$ are generators of $M_\phi$ as an $\OO_\phi[[H]]$-module. $0=f(s_im_i)=s_if(m_i)$ tells us that $f(m_i)=0$. Hence, $f=0$. Therefore, $\Hom_{\OO_\phi[[H]]}(M_\phi,\OO_\phi[[H]])=0$ for any specialization $\phi$ such that $\phi(a)\neq 0$.
\end{proof}

\begin{proof}[Proof of Theorem \ref{generalize Jha}]
    By Theorem \ref{generic quotient}, the module $Y_S(\mathbb A/L)\otimes_\I \OO_\phi$ is pseudo-null over $\OO_\phi[[G]]$ for Zariski generic specializations $\phi$. 
    By Theorem \ref{generic control}, $Y_S(\mathbb A_\phi/L) $ is pseudo-null over $\OO_\phi[[G]]$ for Zariski generic specializations $\phi$.
\end{proof}

Although the main goal of this paper is to study the descent of pseudo-nullity, for comparison with \cite[Theorem~10]{Jha2012}, we also prove the complementary ascent result.

\begin{lemma}\label{lemma complementary}
          Assume $\I$ is a commutative complete regular local ring of characteristic zero with a finite residue field of characteristic $p$ and $G$ is a compact pro $p$-adic analytic group without $p$-torsion. Let $H$ be a closed normal subgroup of $G$ such that $G/H\cong \Z_p$. 

            Assume $\I[[G]]$-module $M$ is finitely generated over $\I[[H]]$. If $M_{\phi_0}:=M\otimes_\I \OO_{\phi_0}$ is a pseudo-null $\OO_{\phi_0}[[G]]$-module for a specializations $\phi_0$, then  $M$ be a pseudo-null $\I[[G]]$-module.
\end{lemma}
\begin{proof}
By remark \ref{skew}, we need to show $M$ is torsion $\I[[H]]$-module. 
\cite[Corollary 4.13]{LimFineSelmer} tells us that if the ring $R$ and its quotient $R/(x)$ for an element $x\in R$ both are complete regular local ring and $M$ is finitely generated over $R[[H]]$,  then the quotient $M/xM$ is torsion over $(R/(x))[[H]]$ implies that $M$ is torsion over $R[[H]]$. 

However, our image of the quotient $\phi_0(\I)$ may not be regular. Fortunately, by the Cohen factorization theorem
\cite[Theorem~1.1]{AvramovFoxbyHerzog1994}, the local homomorphism
$\phi:\mathbb I\longrightarrow\mathcal O_\phi$
admits a factorization
\[
\mathbb I\longrightarrow\mathbb J
\longrightarrow\mathcal O_\phi,
\]
where $\mathbb J$ is a complete Noetherian local ring,
$\mathbb I\to\mathbb J$ is flat with regular closed fiber, and
$\mathbb J\to\mathcal O_\phi$ is surjective. By \cite[\href{https://stacks.math.columbia.edu/tag/031E}{Lemma 031E}]{stacks-project} and $\I$ is regular, we know that $\mathbb{J}$ is complete regular local ring. Since the second map is surjective between regular local rings, by \cite[\href{https://stacks.math.columbia.edu/tag/00NR}{Lemma 00NR}]{stacks-project} and \cite[\href{https://stacks.math.columbia.edu/tag/00NQ}{Lemma 00NQ}]{stacks-project}, we can choose parameters $x_1,x_2,\cdots,x_r\in\ker(\mathbb J\to\mathcal O_\phi)$ such that each quotient $\mathbb{J}_i=\mathbb{J}/(x_1,x_2,\cdots,x_i)$ is complete regular local ring. Hence, by applied \cite[Corollary 4.13]{LimFineSelmer} inductively, we know that $M\widehat\otimes_\I \mathbb{J}$ is a torsion $\mathbb{J}[[H]]$-module.

Since $\mathbb I\to\mathbb J$ is a flat local homomorphism,
it is automatically faithfully flat. In particular, $\mathbb I\to\mathbb J$ is injective. We have $\Hom_{\I[[H]]}(M,\I[[H]])\hookrightarrow \Hom_{\I[[H]]}(M, \mathbb{J}[[H]])\cong \Hom_{\mathbb{J}[[H]]}(M\widehat\otimes_\I \mathbb{J},\mathbb{J}[[H]])=0 $, where the second isomorphism is by adjunction. Therefore, $M$ is a torsion $\I[[H]]$-module and hence, a pseudo-null $\I[[G]]$-module.

\end{proof}

\begin{remark}
The implication $(3)\Rightarrow(1)$ in \cite[Theorem~10]{Jha2012} shows, under certain hypotheses, that if the dual fine Selmer group $Y_S(\mathbb A_{\phi_0}/L)$
is pseudo-null over $\OO_{\phi_0}[[G]]$ for some specialization $\phi_0$, then $Y_S(\mathbb A/L)$
is pseudo-null over $\I[[G]]$. Lemma~\ref{lemma complementary} shows that Jha's additional hypotheses are not needed for the purely algebraic step of lifting pseudo-nullity from a specialization to the big module. As explained in Remark~\ref{jha comment}, the obstruction to removing $(\mathrm{Dim}_S)$ instead occurs in the control theorem comparing
$Y_S(\mathbb A/L)\otimes_{\I}\OO_{\phi_0}$ and $ 
Y_S(\mathbb A_{\phi_0}/L)$.
However, this does not imply that the implication $(3)\Rightarrow(1)$ in
\cite[Theorem~10]{Jha2012} fails without $(\mathrm{Dim}_S)$. Indeed, even if $Y_S(\mathbb A/L)\otimes_{\I}\OO_{\phi_0}$ is not pseudo-null, it is still possible that $Y_S(\mathbb A/L)$ is pseudo-null over $\I[[G]]$. Therefore, proving the general implication $(3)\Rightarrow(1)$ without $(\mathrm{Dim}_S)$ requires a different argument.
\end{remark}
Theorem \ref{generalize Jha} is the generalization of implication $(1)\Rightarrow(2)$ in \cite[Theorem~10]{Jha2012}. Now we obtain the converse direction as a corollary of Lemma \ref{lemma complementary}. 
\begin{corollary}\label{main converse}
        Let $ L/K$ be a Galois extension inside $G_{K,S}$ containing the cyclotomic $\Z_p$-extension $K^{cyc}/K$ such that $G:=\Gal(L/K)$ is a compact pro $p$-adic Lie group without $p$-torsion. Denote $H:=\Gal(L/K^{cyc})$.

Assume the Iwasawa module of the dual fine Selmer group $Y_S(\mathbb A/L)$ is finitely generated over $\I[[H]]$. 
 Assume there exists a Zariski dense subset $U\subset \Spec(\I)$ such that for any specialization $\phi$ with $\PP_\phi=\ker\phi\in U $, the Iwasawa module of the dual fine Selmer group $Y_S(\mathbb{A}_\phi/L)$ is pseudo-null over $\OO_\phi[[G]]$.  Then  $Y_S(\mathbb A/L)$ is pseudo-null over $\I[[G]]$.
\end{corollary}

\begin{proof}
   The intersection of a dense subset with a Zariski open subset in $\Spec(\I)$ is nonempty.  Theorem \ref{generic control} tells us $Y_S(\mathbb A/L)\otimes_{\I}\OO_{\phi}$  is pseudo-null over $\OO_\phi[[G]]$ for a specialization $\phi$. By Lemma \ref{lemma complementary}, $Y_S(\mathbb A/L)$ is pseudo-null over $\I[[G]]$.
\end{proof}

\subsection{Back to Hida's family}\label{back to Jha}

For the Hida family, we assume $\I=\OO[[W]]$ as in \cite{Jha2012}, where $\OO$ is the ring of integers of a local field extension over $\Q_p$ and $W$ is a variable. The kernel $\PP=\ker\xi$ of the arithmetic specialization is a height-one prime ideal in $ \OO[[W]]$ and hence is principal. We may write $\PP=(T_\xi)$.

 We can use Theorem \ref{non ht=2} to give a different proof of Jha's result and give an explicit description of the exceptional locus, with more detail if $\I[[G]]$ is abelian. 

\begin{proposition}\label{Jha}
        Assume $\I$ is a commutative complete regular local ring of characteristic zero with a finite residue field of characteristic $p$ and $G$ is a compact pro $p$-adic analytic group without $p$-torsion. Let $H$ be a closed normal subgroup of $G$ such that $G/H\cong \Z_p$. 
    
    Assume $M$ is a pseudo-null $S:=\I[[G]]$-module and is finitely generated over $R:=\I[[H]]$. Then $M/(M\PP)$ is pseudo-null over $S/(\PP S)$ for all but finitely many height-one prime ideals $\PP$ of $\I$. The exceptional locus are those for which $\Ext^2_{S}(M,S)[\PP]\neq 0$.

\end{proposition}

\begin{proof}
    By \cite[Section 4]{SchneiderVenjakob2006}, $S$ is a skew power series ring over $R$. If $N$ is an $S$-module and is finitely generated over $R$, then \cite[Prop 3.1]{SchneiderVenjakob2006} tells us $\Ext^i_{S}(N,S)={}^\sigma\Ext^{i-1}_{R}(N,R)$ and $\Hom_S(N,S)=0$.
    Hence, $\Ext^2_{S}(M,S)$ is also a finitely generated $R$-module.

Since $S$ is an Auslander regular ring and $M$ is pseudo-null over $S$, we have $j_S(M)\geq 2$. If $j_S(M)\geq 3$, then $E:=\Ext^2_{S}(M,S)=0$, so the conclusion holds. Now assume $j_S(M)=2$. By \cite[Prop 3.5(iii)(c)]{Venjakob2002}, we have that $E$ is pure of grade 2 over $S$. Hence, for any $S$-submodule $N\subset E$, $j_S(N)=2$. By \cite[Prop 3.2]{SchneiderVenjakob2006}, $j_R(N)=j_S(N)-1=1$. Hence, $\Ext^1_R(N,R)\neq 0$ for any nonzero $S$-submodule $N$ of $E$. 

Since a height-one prime ideal of $\I$ is principal, we can use Theorem \ref{non ht=2}. Assume there are infinitely many height-one primes $(\PP_i)=(T_i)$ of $\I$ such that $E_i:=E[\PP_i]=\Ext^2_S(M,S)[\PP_i]\neq 0$. Set $N=E_n\cap (E_1+E_2+\cdots+E_{n-1})$. We will show $N=0$. Set $Q=T_1T_2\cdots T_{n-1}$. Then $QN=0$ and $\PP_nN=0$. Since $Q\not\in \PP_n$, the image of $Q$ in $\I/\PP_n$ is a nonzero divisor. Therefore, the image of $Q$ in $R/(T_nR)$ is a nonzero divisor. We have $\Hom_{R/(T_nR)}(N,R/(T_nR))=0$. By Lemma \ref{non quotient}, we know $\Ext^1_R(N,R)\cong \Hom_{R/(T_nR)}(N,R/(T_nR))=0$. Since $\Ext^1_R(N,R)\neq 0$ for any nonzero $S$-submodule $N$ of $E$, we have $N=E_n\cap (E_1+E_2+\cdots+E_{n-1})=0$. Then 
\[
E_1\subsetneq E_1\oplus E_2\subsetneq E_1\oplus E_2\oplus E_3 \subsetneq \cdots 
\]
would be an infinite ascending chain of \(R\)-submodules of \(E\), contradicting the fact that $E$ is a finitely generated module over the Noetherian ring \(R\).

\end{proof}
\begin{remark}
   Applying Proposition \ref{Jha} to Jha's setting, we find that $Y(\mathcal{F}/G)/(T_\xi Y(\mathcal{F}/G))$ is pseudo-null for all but finitely many arithmetic specializations $\xi$, where $\ker\xi=(T_\xi)$. 
\end{remark}
Proposition \ref{Jha} also tells us that the exceptions occur at $\PP$ such that $\Ext^2_{S}(M,S)[\PP]\neq 0$. If $G\cong \Z_p^d$, $\I[[G]]$ becomes a commutative ring, and we can give a more explicit description of the exceptional locus.
\begin{corollary}
    Under the same setup as in Proposition \ref{Jha} and further assuming $G\cong \Z_p^d$, each height-2 associated prime ideal of $M$ contains at most one $\PP$ such that $M/(M\PP)$ is not pseudo-null over $S/(\PP S)$.
\end{corollary}
\begin{remark}
    The corollary tells us that the number of exceptional cases is at most the number of height-2 associated prime ideals of $M$. The result is more precise and concrete.
\end{remark}
\begin{proof}
    By Corollary \ref{compare}, the exceptional case $\Ext^2_{S}(M,S)[\PP]\neq 0$ occurs if and only if the ideal $\PP=(T)$ is contained in some height-2 associated prime ideal $P$ of $M$. Since $P$ is an associated prime of $M$, we have $S/P\hookrightarrow M$. Since $M$ is finitely generated as an $R$-module, $S/P$ is finitely generated as an $R/(P\cap R)$-module. We have $\dim S/P=\dim R/(P\cap R)$. On the one hand, $\dim S/P=\dim S-\Ht(P)=\dim S-2=\dim R-1$. On the other hand, $\dim R/(P\cap R)=\dim R-\Ht_R(P\cap R)$. Hence, $\Ht(P\cap R)=1$. Since $T$ is a regular element, $\Ht_R(\PP R)\geq 1 $. Since $\PP R\subset P\cap R$, we have $P\cap R=\PP R$. Hence, a height-2 associated prime ideal $P$ determines $\PP $ uniquely.  
\end{proof}

\printbibliography
\end{document}